\documentclass[11pt]{article}

\usepackage{amsfonts}
\usepackage{amssymb}
\usepackage{amsmath}
\usepackage{amsthm}

\def\negthickspace{\!\!\!}

\newcommand{\nicefrac}[2]
{\leavevmode \kern.1em\raise.5ex\hbox{\the\scriptfont0 #1}
             \kern-.1em/\kern-.15em\lower.25ex
             \hbox{\the\scriptfont0 #2}}

\newtheorem{theorem}{Theorem}
\newtheorem{proposition}{Proposition}
\newtheorem{corollary}{Corollary}
\newtheorem*{problem}{Problem}
\newtheorem{definition}{Definition}
\newtheorem{lemma}{Lemma}

\theoremstyle{definition}
\newtheorem{remark}{Remark}

\theoremstyle{definition}

\begin{document} 

\begin{center}
{\Large{\sc On the existence of normal Coulomb frames}}\\[1ex]
{\Large{\sc for two-dimensional immersions with higher codimension}}\\[3ex]
{\large Steffen Fr\"ohlich, Frank M\"uller}\\[0.4cm]
{\small\bf Abstract}\\[0.4cm]
\begin{minipage}[c][4cm][l]{12cm}
{\small In this paper we consider the existence and regularity problem for
  Coulomb frames in the normal bundle of two-dimensional surfaces with higher
  codimension in Euclidean spaces. While the case of two codimensions can be
  approached directly by potential theory, more sophisticated methods have to
  be applied for codimensions greater than two. As an application we include an a priori estimate for the corresponding torsion coefficients in arbitrary
  codimensions.}
\end{minipage}
\end{center}
{\small MCS 2000: 35J55, 49N60, 53A07}
\section{Introduction}
This paper is devoted to the analysis of {\it normal Coulomb frames} for two-dimensional surfaces with higher codimensions $n\ge 2$ in Euclidean spaces $\mathbb R^{n+2}.$ It is the third part of a sequence of papers on embedding problems for surfaces in Euclidean spaces.\\[1ex]
In \cite{froehlich_mueller_2007} we fully treated the case $n=2$ of two codimensions. Then we investigated normal bundles of surfaces with arbitrary codimension $n>2$  in \cite{froehlich_mueller_2008}. In particular, we focused on {\it torsions associated with normal frames,} already introduced by Weyl in \cite{weyl_1921}, and we presented various ways to control analytically the torsion coefficients of so-called {\it normal Coulomb frames} which are crititcal for a {\it functional of total torsion.} But existence and regularity of such frames are left open in this second paper.\\[1ex]
In the present paper we thus prove existence and classical regularity of normal Coulomb frames for surfaces in Euclidean spaces $\mathbb R^{n+2}$ with arbitrary codimensions $n\ge 2.$
\subsection{Basic definitions}
Let us start with some basic definitions: For an integer $n\ge 1$ we consider vector-valued mappings\footnote{Any vector $Z\in\mathbb R^{n+2}$ represents a \emph{column vector}, we write $Z=(z^1,\ldots,z^{n+2})$ only for visual reasons. \emph{Row vectors} are then denoted by $Z^t$.}
  $$X=X(w)=\big(x^1(u,v),\ldots,x^{n+2}(u,v)\big),\quad w=(u,v)\in\overline B,$$
defined on the closure of the open unit disc $B:=\{w\in\mathbb R^2\,:\,|w|< 1\}.$ Suppose $X\in C^{k,\alpha}(\overline B,\mathbb R^{n+2})$ with an integer $k\ge 3$ and with some $\alpha\in(0,1)$. In combination with the geometric regularity condition
  $$\mbox{rank}\,DX(w)=2\quad\mbox{for all}\ w\in\overline B$$
for the Jacobian $DX\in\mathbb R^{(n+2)\times 2},$ the mapping $X$ represents a {\it regular surface} or a {\it two-dimensional immersion of disc-type.} The linearly independent tangential vectors $X_u=\frac\partial{\partial u}X$ and $X_v=\frac\partial{\partial v}X$ span the {\it tangential space} $\mathbb T_X(w)$ at the particular point $w=(u,v)\in\overline B,$ i.e.
  $$\mathbb T_X(w)=\mbox{span}\,\big\{X_u(w),X_v(w)\big\}.$$
For the whole paper we assume $X$ to be conformally parametrized, i.e.~the
conformality relations
  $$g_{11}=g_{22}>0,\quad g_{12}=0\quad\mbox{on}\ \overline B$$
are satisfied for the coefficients $g_{ij}=\langle X_{u^i},X_{u^j}\rangle$ of the first fundamental form of $X.$\\[1ex] 
\noindent
Next, we define the {\it normal (co-)space} ($\langle\cdot,\cdot\rangle$ denotes the inner product between two vectors)
  $$\mathbb N_X(w):=\mathbb T_X(w)^\perp=\big\{Z\in\mathbb R^{n+2}\,:\,\langle X_u(w),Z\rangle=\langle X_v(w),Z\rangle=0\big\}
    \quad\mbox{for}\ w\in\overline B.$$
Then we have the decomposition of the ambient space $\mathbb R^{n+2}=\mathbb T_X(w)\oplus\mathbb N_X(w)$ for each $w\in\overline B$. We choose unit normal vectors $N_\sigma=N_\sigma(w),$ $\sigma=1,\ldots,n,$ satisfying $\langle N_\sigma,N_\vartheta\rangle=\delta_{\sigma\vartheta}$, where $\delta_{\sigma\vartheta}=\delta_\sigma^\vartheta=\delta^{\sigma\vartheta}$ denote the usual Kronecker symbols, spanning $\mathbb N_X(w)$ and being oriented:
\begin{equation}\label{1.1}
 \det\big(X_u,X_v,N_1,\ldots,N_n\big)>0.
\end{equation}
\begin{definition}
A matrix $N=(N_1,\ldots,N_n)\in C^2(\overline B,\mathbb R^{(n+2)\times n})$ consisting of $n\ge 1$ orthonormal unit normal vectors $N_\sigma=N_\sigma(w)$, being oriented in the sense of (\ref{1.1}) and spanning the $n$-dimensional normal space $\mathbb N_X(w)$ at each point $w\in \overline B$, is called a \emph{normal frame.}
\end{definition}
\noindent 
Finally, we sometimes interpret a matrix $A=(A_\sigma^\vartheta)_{\sigma,\vartheta=1,\ldots,n}\in\mathbb R^{n\times n}$ as a vector in $\mathbb R^{n^2}$ assigning the well-known scalar product and length
  $$\langle A,B\rangle:=\mbox{trace}(A\circ B^t)=\sum_{\sigma,\vartheta=1}^n A_\sigma^\vartheta B_\sigma^\vartheta,\quad
    |A|:=\sqrt{\langle A,A\rangle}=\bigg(\sum_{\sigma,\vartheta=1}^n (A_\sigma^\vartheta)^2\bigg)^{\frac12}.$$
\subsection{Normal Coulomb frames}
For a given surface $X\in C^{k,\alpha}(\overline B,\mathbb R^{n+2})$ there always exists a normal frame $N\in C^{k-1,\alpha}(\overline B)$, \emph{but its choice is not unique!} Rather, we can transform a given normal frame $N=(N_1,\ldots,N_n)$ by means of an orthogonal mapping $(R_\sigma^\vartheta)_{\sigma,\vartheta=1,\ldots,n}\in C^2(\overline B,SO(n))$ into a new normal frame $\widetilde N=(\widetilde N_1,\ldots,\widetilde N_n)$ as follows:
\begin{equation}\label{drehung}
  \widetilde N_\sigma=\sum_{\vartheta=1}^nR_\sigma^\vartheta N_\vartheta\,,\quad\sigma=1,\ldots,n.
\end{equation}
According to this freedom of choice there naturally arises the question: What is a ``good'' frame? Possibly there exists a {\it parallel frame} $N$ for the given surface $X.$ {\it Parallelity} in this context means that all derivatives of any unit normal vector $N_\sigma$ are tangential.\\[1ex]
It turns out that {\it parallel frames are special normal Coulomb frames.} To see this, let us specify the term \emph{normal Coulomb frame}: First, let $u^1:=u$ and $u^2:=v$ just to make the Ricci calculus applicable. We introduce the {\it connection coefficients of the normal bundle}, also called {\it torsion coefficients}\footnote{Compare these quantities with the torsion $\tau$ of an arc-length parametrised curve in $\mathbb R^3,$ equipped with the standard frame $\{{\mathfrak t},{\mathfrak n},{\mathfrak b}\}$ consisting of the unit tangential vector ${\mathfrak t},$ the unit normal vector ${\mathfrak n},$ and the unit binormal vector ${\mathfrak b}.$ Then, the identity $\tau={\mathfrak n}'\cdot{\mathfrak b}=-{\mathfrak n}\cdot{\mathfrak b}'$ justifies our notation  {\it torsion coefficient} for $T_{\sigma,i}^\vartheta.$\\
As already mentioned, torsion coefficients for orthonormal frames in the normal bundles of higher-dimensional manifolds in Euclidean spaces were first considered by Weyl \cite{weyl_1921}.} (see e.g. \cite{chavel_2006} Section II.2 or \cite{helein_2002} pp.~61--63),
\begin{equation}\label{1.2}
  T_{\sigma,i}^\vartheta
  :=\langle N_{\sigma,u^i},N_\vartheta\rangle
   =-\langle N_\sigma,N_{\vartheta,u^i}\rangle=
   -T_{\vartheta,i}^\sigma\,,\quad
  i=1,2,\ \sigma,\vartheta=1,\ldots,n,
\end{equation}
where $N_{\sigma,u^i}=\frac{\partial}{\partial u^i}\,N_\sigma$ etc. Then a normal frame $N$ is parallel if and only if it is {\it free of torsion}, i.e. if all torsion coefficients $T_{\sigma,i}^\vartheta$ vanish identically.
\\[1ex]
\noindent
Next, let us write $g^{ij}$ for the coefficients of the inverse
$(g_{ij})_{i,j=1,2}^{-1}$ of the metric tensor and $W:=\sqrt{g_{11}g_{22}-g_{12}^2}$ for the area element of the surface.  
\begin{definition}\label{def_Coulomb}
A \emph{normal Coulomb frame} is a normal frame which is critical for the \emph{functional of total torsion}
  $${\mathcal T}_X(N)
    =\sum_{i,j=1}^2\,
     \sum_{\sigma,\vartheta=1}^n\,
     \int\hspace{-1.5ex}\int\limits_{\hspace{-2ex}B}
     g^{ij}T_{\sigma,i}^\vartheta T_{\sigma,j}^\vartheta\,W\,dudv.$$
Here a normal frame $N$ is called \emph{critical} for $\mathcal T_X$ if the first variation $\lim_{\varepsilon\to0}\frac1\varepsilon\{\mathcal T_X(\widetilde N)-\mathcal T_X(N)\}$ vanishes for all normal frames $\widetilde N=(\widetilde N_1(w,\varepsilon),\ldots,\widetilde N_n(w,\varepsilon))$ defined by
  $$\widetilde N_\sigma(w,\varepsilon)=\sum_{\vartheta=1}^nR_\sigma^\vartheta(w,\varepsilon)N_\vartheta(w),\quad w\in\overline B,
    \quad \varepsilon\in(-\varepsilon_0,\varepsilon_0),$$
with small $\varepsilon_0>0$ and a one-parameter family $(R_\sigma^\vartheta(w,\varepsilon))_{\sigma,\vartheta=1,\ldots,n}\in C^2(\overline B\times(-\varepsilon_0,+\varepsilon_0),SO(n))$ satisfying $R_\sigma^\vartheta(w,0)=\delta_\sigma^\vartheta$ on $\overline B$.
\end{definition}
\noindent
The functional of total torsion {\it does not depend on the special parameters} $u^i.$ But taking the conformal parametrization and the skew symmetry of the torsion coefficients in $\sigma\leftrightarrow\vartheta$ into account, it takes the particularly simple form
  $${\mathcal T}_X(N)
    =\sum_{\sigma,\vartheta=1}^n\int\hspace{-1.5ex}\int\limits_{\hspace{-2ex}B}
      \Big\{
        (T_{\sigma,1}^\vartheta)^2+(T_{\sigma,2}^\vartheta)^2
      \Big\}\,dudv = 2\int\hspace{-1.5ex}\int\limits_{\hspace{-2ex}B}
      \Big\{
        (T_{1,1}^2)^2+(T_{1,1}^3)^2+\ldots+(T_{n-1,2}^n)^2
      \Big\}\,dudv\,.$$
There always holds ${\mathcal T}_X(N)\ge 0$, and we have ${\mathcal T}_X(N)= 0$ if and only if $N$ is parallel. Hence, parallel frames are special normal Coulomb frames. One of our results in \cite{froehlich_mueller_2007}, \cite{froehlich_mueller_2008} is \emph{that any normal Coulomb frame is parallel if the normal bundle of the immersion is flat.} In other words, if all components $S_{\sigma,ij}^\vartheta$ of the following \emph{curvature tensor of the connection coefficients $T_{\sigma,i}^\vartheta$} vanish identically (see also \cite{chen_1973} or \cite{helein_2002}, and the references therein):
\begin{equation}\label{normal_curvature_tensor}
 S_{\sigma,ij}^\vartheta
  :=T_{\sigma,i,u^j}^\vartheta-T_{\sigma,j,u^i}^\vartheta
   +\sum_{\omega=1}^n\big(T_{\sigma,i}^\omega T_{\omega,j}^\vartheta-T_{\sigma,j}^\omega T_{\omega,i}^\vartheta\big)\,,\quad
  i,j=1,2,\ \sigma,\vartheta=1,\ldots,n.
\end{equation}
This special property of the $T_{\sigma,i}^\vartheta$ does not depend on the choice of the normal frame and the parameters $u^i$ of $X$. But in  the general non-vanishing case, the $S_{\sigma,ij}^\vartheta$ depend on the chosen normal frame and the parametrization of $X$. We will address the issue of the link between $S_{\sigma,ij}^\vartheta$ and a geometric curvature quantity of the normal bundle in Subsection \ref{subsection_4.1}.\\[1ex]
The curvature tensor (\ref{normal_curvature_tensor}) is shortly named \emph{normal curvature tensor}. It emerges from the so-called {\it Ricci integrability conditions} which demand the vanishing of the normal components of the identity $N_{\sigma,u^iu^j}-N_{\sigma,u^ju^i}\equiv 0$ right in a similar way as the Riemannian curvature tensor $R_{ijk\ell}$ emerges from the integrability conditions w.r.t. $X_{u^iu^ju^k}-X_{u^iu^ku^j}\equiv 0$ for tan\-gential vector fields $X_{u^i}.$ \\[1ex]
In the general case of nonvanishing normal curvature tensor, there cannot exist a parallel frame in the normal bundle. Then the torsion coefficients  appear explicitely in many equations of the differential geometry. For instance, they can be found in the second variation formula of the area functional. Hence, one has to control the torsion coefficients if one wants to prove geometric estimates depending on stability questions. And since normal Coulomb frames are critical points of the $L^2$-norm of the torsion coefficients by Definition\,\ref{def_Coulomb}, we are led to the study of normal Coulomb frames instead of parallel frames in the sequel.\\[1ex]
We finally note that definition (\ref{normal_curvature_tensor}) yields that the normal curvature tensor is completely determined by the coefficients\footnote{In fact, we have also a skew symmetry in $\sigma\leftrightarrow\vartheta$, but we will not use this property in the present paper.
}
\begin{equation}\label{s_12}
  S_{\sigma,12}^\vartheta
  :=T_{\sigma,1,u^2}^\vartheta-T_{\sigma,2,u^1}^\vartheta
   +\sum_{\omega=1}^n\big(T_{\sigma,1}^\omega T_{\omega,2}^\vartheta-T_{\sigma,2}^\omega T_{\omega,1}^\vartheta\big)\,,\quad
  \sigma,\vartheta=1,\ldots, n\,.
\end{equation}
\subsection{Overview}
In this paper we consider the following
\begin{problem}
Let the conformally parametrized immersion $X\in C^{k,\alpha}(\overline B,\mathbb R^{n+2})$ with $k\ge 3$, $\alpha\in(0,1)$ be given. Does there always exist a smooth normal Coulomb frame $N$, i.e.~a normal frame which is critical for the functional of total torsion ${\mathcal T}_X$ and belongs to class $C^{k-1,\alpha}(\overline B)$? 
\end{problem}
\noindent
In the subsequent sections we answer this question by YES; this is the content of Theorems \ref{existence_r4} and \ref{existence_rn} below. We can even choose $N$ to be minimizing for $\mathcal T_X$.\\[1ex]
Obviously, the case $n=1$ is trivial. In case of codimension $n=2$ we best let classical potential theory work to solve our problem (see Section \ref{section_3} below). For $n\ge3$ we need a more subtle approach: In Subsection \ref{ss4.1} we first construct a \emph{weak} normal Coulomb frame $N$ of class $W^{1,2}(B)\cap L^\infty(B)$ by employing a variational argument which goes back to Fr\'ed\'eric H\'elein \cite{helein_2002}: To study harmonic mappings into Riemannian manifolds of arbitrary dimension and without special geometric symmetries, H\'elein introduced \emph{tangential} Coulomb frames (i.e.~special sections of the orthogonal \emph{tangential} frame bundle) as critical points of an energy functional similar to our functional of total torsion.\\[1ex]
In a second step we prove that the constructed weak normal Coulomb frame is in fact smooth (Subsection\,\ref{ss4.2}). This investigation is quite different from H\'elein's analysis, who was interested in the regularity of the harmonic map itself rather than the Coulomb frame. The plan of our proof is as follows: As in \cite{froehlich_mueller_2008} (and as also done by H\'elein), we interpret the Euler-Lagrange equations for our Coulomb frame as integrability conditions. Then the special structure of the torsion coefficients yields a Poisson system for an integral function emerging from Poincare's lemma with right-hand side of div-curl type along with a homogeneous Dirichlet boundary condition. From Wente's inequality we then obtain $N\in W^{2,1}_{loc}(B)$ for our weak normal Coulomb frame, see Lemma \ref{w_12}. On the other hand, (\ref{normal_curvature_tensor}) gives a nonlinear, inhomogeneous second-order system with div-curl structure for that integral function. Now it is important to observe that the zero-order term $S_{12}=(S_{\sigma,12}^\vartheta)_{\sigma,\vartheta}$ of that system is length-invariant under rotations $(R_\sigma^\vartheta)_{\sigma,\vartheta}\in W^{2,1}_{loc}(B,SO(n))$, and this is what we elaborate in Subsection 4.1. This enables us to apply another part of Wente's inequality to obtain global continuity for the integral function and, as a consequence, even $C^{1,\alpha}(\overline B)$-regularity. Now an interplay between the mentioned nonlinear second-order system and the Weingarten equations for our surface $X$ yields the desired regularity result Theorem \ref{existence_rn}.\\[1ex]
We close with a new a priori estimate of the torsion coefficients for a conformally parametrized immersion $X:\overline B\to\mathbb R^{n+2}$ in terms of the normal curvature tensor. 
\section{Euler-Lagrange equations for Coulomb frames}
\label{section_2}
\setcounter{equation}{0}
We briefly repeat the calculation for the Euler-Lagrange equations for ${\mathcal T}_X$-critical normal frames from \cite{froehlich_mueller_2008}:  Consider  an arbitrary family of rotations
  $$R=R(w,\varepsilon)
    =\big(R_\sigma^\vartheta(w,\varepsilon)\big)_{\sigma,\vartheta=1,\ldots,n}
    \in C^2\big(\overline B\times(-\varepsilon_0,+\varepsilon_0),SO(n)\big)$$
with small $\varepsilon_0>0,$ satisfying $R(w,0)=\mbox{id}$ and $\frac{\partial}{\partial\varepsilon}R(w,0)=A(w)\in C^1(\overline B,so(n))$, where $so(n)$ denotes the Lie algebra for $SO(n),$ and $\delta_\sigma^\vartheta$ are the Kronecker symbols. Thus it holds
\begin{equation*}
  R(w,\varepsilon)=\mbox{id}+\varepsilon A(w)+o(\varepsilon)\,.
\end{equation*}
We apply $R$ to a given normal frame $N$ and deduce
\begin{equation*}
  \widetilde N_\sigma
  =\sum_{\vartheta=1}^n
   R_\sigma^\vartheta N_\vartheta
  =N_\sigma
   +\varepsilon
    \sum_{\vartheta=1}^n
    A_\sigma^\vartheta N_\vartheta
   +o(\varepsilon)\,.
\end{equation*}
Consequently, the new torsion coefficients can be expanded to
\begin{equation*}
\begin{array}{l}
  \displaystyle
  \widetilde T_{\sigma,i}^\omega=\langle\widetilde N_{\sigma,u^i},\widetilde N_\omega\rangle
  \,=\,T_{\sigma,i}^\omega
       +\varepsilon A_{\sigma,u^i}^\omega
       +\varepsilon\sum_{\vartheta=1}^n\big(
          A_\sigma^\vartheta T_{\vartheta,i}^\omega
          +A_\omega^\vartheta T_{\sigma,i}^\vartheta
       \big)
       +o(\varepsilon)\,, \\[2ex]
  \displaystyle
  (\widetilde T_{\sigma,i}^\omega)^2
  \,=\,(T_{\sigma,i}^\omega)^2
       +2\varepsilon\,
       \Big\{
         A_{\sigma,u^i}^\omega T_{\sigma,i}^\omega
         +\sum_{\vartheta=1}^n\big(
            A_\sigma^\vartheta T_{\vartheta,i}^\omega T_{\sigma,i}^\omega
            +A_\omega^\vartheta T_{\sigma,i}^\vartheta T_{\sigma,i}^\omega
         \big)
       \Big\}
       +o(\varepsilon)\,.
\end{array}
\end{equation*}
Employing the skew-symmetry of $A=(A_\sigma^\vartheta)_{\sigma,\vartheta}$ and $T_{\sigma,i}^\omega$, we find 
\begin{equation*}
  \sum_{\sigma,\omega,\vartheta=1}^n
  \big(
    A_\sigma^\vartheta T_{\vartheta,i}^\omega T_{\sigma,i}^\omega
    +A_\omega^\vartheta T_{\sigma,i}^\vartheta T_{\sigma,i}^\omega
  \big)
  =
  \sum_{\sigma,\omega,\vartheta=1}^n
  \big(
    A_\sigma^\vartheta T_{\vartheta,i}^\omega T_{\sigma,i}^\omega
	-A_\sigma^\vartheta T_{\omega,i}^\sigma T_{\omega,i}^\vartheta
  \big)
  =0.
\end{equation*}
Hence, summing up $(\widetilde T_{\sigma,i}^\omega)^2$ over $\sigma,\omega=1,\ldots,n$ and $i=1,2$ and integrating over $B$, we arrive at
\begin{equation*}
\begin{array}{l}
  \displaystyle
  {\mathcal T}_X(\widetilde N)-{\mathcal T}_X(N)
  \,=\,2\varepsilon
       \sum_{\sigma,\omega=1}^n\sum_{i=1}^2\,
       \int\hspace{-1.3ex}\int\limits_{\hspace{-1.8ex}B}
       A_{\sigma,u^i}^\omega T_{\sigma,i}^\omega\,dudv
       +o(\varepsilon) 
  \\[4ex]\hspace*{8ex}\displaystyle
  =\,2\varepsilon\sum_{\sigma,\omega=1}^n
     \int\limits_{\partial B}
     A_\sigma^\omega\big\langle(T_{\sigma,1}^\omega,T_{\sigma,2}^\omega),\nu\big\rangle\,ds
     -2\varepsilon\sum_{\sigma,\omega=1}^n
       \int\hspace{-1.3ex}\int\limits_{\hspace{-1.8ex}B}
       A_\sigma^\omega\,\mbox{div}\,(T_{\sigma,1}^\omega,T_{\sigma,2}^\omega)\,dudv
       +o(\varepsilon)
\end{array}
\end{equation*}
with the outward unit normal $\nu$ on $\partial B.$ 
\begin{proposition}\label{p1}
(Fr\"ohlich, M\"uller \cite{froehlich_mueller_2008})\\
$N$ is a normal Coulomb frame, i.e.~$N$ is critical for the total torsion ${\mathcal T}_X,$ if and only if there hold
\begin{equation}\label{euler_lagrange}
  \mbox{\rm div}\,(T_{\sigma,1}^\vartheta,T_{\sigma,2}^\vartheta)=0\quad\mbox{in}\ B,\quad
  \big\langle(T_{\sigma,1}^\vartheta,T_{\sigma,2}^\vartheta),\nu\big\rangle=0\quad\mbox{on}\ \partial B
\end{equation}
for all $\sigma,\vartheta=1,\ldots,n.$
\end{proposition}
\begin{remark}\label{weak_coulomb}
Note that the above computations are meaningful also for \emph{weak normal Coulomb frames} $N\in W^{1,2}(B)\cap L^\infty(B)$. The respective torsion coefficients $T_{\sigma,i}^\vartheta\in L^2(B)$, $\sigma,\vartheta=1,\ldots,n$, $i=1,2,$ are then weak solutions of (\ref{euler_lagrange}), that means, for any choice of $\sigma,\vartheta\in\{1,\ldots,n\}$ we have
\begin{equation}\label{weak_euler}
  \int\hspace{-1.3ex}\int\limits_{\hspace{-1.8ex}B}\big\{\varphi_{u^1}T_{\sigma,1}^\vartheta+\varphi_{u^2}T_{\sigma,2}^\vartheta\big\}\,dudv=0
  \quad\mbox{for all}\ \varphi\in C^\infty(\overline B).
\end{equation}
\end{remark}
\section{Surfaces in $\mathbb R^4$}
\label{section_3}
\setcounter{equation}{0}
In contrast to the previous section we now want to transform a given normal frame $\widetilde N=(\widetilde N_1,\widetilde N_2)$ in $\mathbb R^4$ into a ``good'' normal frame $N=(N_1,N_2)$. This can be done by means of a $SO(2)$-action: 
\begin{equation}\label{so2}
  \begin{pmatrix}
    N_1 ,N_2
  \end{pmatrix}
  =\begin{pmatrix}
     \widetilde N_1, \widetilde N_2
  \end{pmatrix}\circ
  \begin{pmatrix}
     \cos\varphi & -\sin\varphi \\[1ex]
     \sin\varphi & \cos\varphi
   \end{pmatrix}
\end{equation}
with a rotation angle $\varphi\in C^2(\overline B,\mathbb R).$ Then the torsion coefficients of both frames are related by the linear system\footnote{It is sufficient to consider $T_{1,1}^2$ and $T_{1,2}^2$; all other torsion coefficients are either zero or agree with one of them up to the sign on account of their skew-symmetry.}
\begin{equation}\label{transformation_r4}
  T_{1,1}^2= \widetilde T_{1,1}^2+\varphi_u\,,\quad
  T_{1,2}^2= \widetilde T_{1,2}^2+\varphi_v\,.
\end{equation}
Assume now $\widetilde N\in C^{k-1,\alpha}(\overline B,\mathbb R^{4\times 2})$. Due to Proposition \ref{p1} and formula (\ref{transformation_r4}), the frame $N$ from (\ref{so2}) is a normal Coulomb frame if and only if there hold
  $$\mbox{div}\,(\widetilde T_{1,1}^2+\varphi_u,\widetilde T_{1,2}^2+\varphi_v)=0\quad\mbox{in}\ B,\quad
    \big\langle(\widetilde T_{1,1}^2+\varphi_u,\widetilde T_{1,2}^2+\varphi_v),\nu\big\rangle=0\quad\mbox{on}\ \partial B\,.$$
Thus we have to solve the Neumann boundary value problem
\begin{equation}\label{neumann_r4} 
  \Delta\varphi=-\mbox{div}\,(\widetilde T_{1,1}^2,\widetilde T_{1,2}^2)=:f\quad\mbox{in}\ B,\quad
  \frac{\partial\varphi}{\partial\nu}=-\big\langle(\widetilde T_{1,1}^2,\widetilde T_{1,2}^2),\nu\big\rangle=:g\quad\mbox{on}\ \partial B\,,
\end{equation}
which has a solution on account of the integrability condition
  $$\int\hspace{-1.5ex}\int\limits_{\hspace{-2ex}B}
    \mbox{div}\,(\widetilde T_{1,1}^2,\widetilde T_{1,2}^2)\,dudv
    =\int\limits_{\partial B}
     \big\langle(\widetilde T_{1,1}^2,\widetilde T_{1,2}^2),\nu\big\rangle\,ds\,.$$
Note that the right-hand sides in (\ref{neumann_r4}) satisfy $f\in C^{k-3,\alpha}(\overline B,\mathbb R)$, $g\in C^{k-2,\alpha}(\partial B,\mathbb R)$.
Thus, classical potential theory yields
\begin{theorem}
\label{existence_r4}
Suppose $X\in C^{k,\alpha}(\overline B,\mathbb R^4)$ with $k\ge 3,$ $\alpha\in(0,1).$ Then there exists a Coulomb frame $N\in C^{k-1,\alpha}(\overline B,\mathbb R^{4\times 2})$ satisfying the Euler-Lagrange system (\ref{euler_lagrange}) and minimizing ${\mathcal T}_X,$ i.e.
  $${\mathcal T}_X(\widetilde N)-{\mathcal T}_X(N)\ge 0$$
for all normal frames $\widetilde N$.
\end{theorem}
\noindent
The minimizing character of a normal Coulomb frame (for codimension $n=2$) can be deduced easily from (\ref{transformation_r4}), see \cite{froehlich_mueller_2007} for details.
\section{Surfaces in $\mathbb R^{n+2}$}
\label{section_4}
\setcounter{equation}{0}
\subsection{Geometry of the normal curvature tensor}\label{subsection_4.1}
Before we come to the promised existence and regularity results for normal Coulomb frames in higher codimensions it is necessary to clarify the nature of the curvature tensor of the normal connection $T_{\sigma,i}^\vartheta.$ For this purpose, we again fix a normal frame $\widetilde N\in C^{k-1,\alpha}(\overline B,\mathbb R^{(n+2)\times n})$ and consider the transformation
\begin{equation}\label{transform_N}
  N_\sigma=\sum_{\vartheta=1}^nR_\sigma^\vartheta\widetilde N_\vartheta
\end{equation}
with some orthogonal mapping $R=(R_\sigma^\vartheta)_{\sigma,\vartheta=1,\ldots,n}\in C^2(\overline B,SO(n)).$
\begin{definition}
We additionally set
  $$T_i=(T_{\sigma,i}^\vartheta)_{\sigma,\vartheta=1,2,\ldots}\,,\quad
    S_{12}=(S_{\sigma,12}^\vartheta)_{\sigma,\vartheta=1,2,\ldots}\,.$$
\end{definition}
\noindent
For $n=2$, the matrix $S_{12}$ can be easily seen to be invariant under rotations. This behaviour of $S_{12}$ changes for higher codimension: It turns out that only the length of $S_{12}$ remains invariant. Our next result contains this $L^\infty$-invariance of $S_{12}$ which is crucial for our main regularity result.
\begin{theorem}
\label{curvature_vector}
For a fixed normal frame $\widetilde N\in C^{k-1,\alpha}(\overline B,\mathbb R^{(n+2)\times n})$ define $N\in C^2(\overline B,\mathbb R^{(n+2)\times n})$ by (\ref{transform_N}) with some rotation $R=(R_\sigma^\vartheta)_{\sigma,\vartheta=1,\ldots,n}\in C^2(\overline B,SO(n)).$ Then there holds
 \begin{equation}\label{transform_S}
  S_{12}=R\circ\widetilde S_{12}\circ R^t
 \end{equation}
for the corresponding curvatures. In particular, the length $|S_{12}|$ is invariant under rotations.
\end{theorem}
\begin{proof}
First we note
  $$\begin{array}{lll}
      \displaystyle
      T_{\sigma,i}^\vartheta\negthickspace
      & = & \negthickspace\displaystyle
            \langle
              N_{\sigma,u^i},N_\vartheta
            \rangle
            \,=\,\Big\langle
                   \sum_{\alpha=1}^n
                   \big(R_{\sigma,u^i}^\alpha\widetilde N_\alpha+R_\sigma^\alpha\widetilde N_{\alpha,u^i}\big)\,,
                   \sum_{\beta=1}^nR_\vartheta^\beta\widetilde N_\beta
                 \Big\rangle \\[3ex]
     & = & \displaystyle\negthickspace
           \sum_{\alpha,\beta=1}^n
           \big(R_{\sigma,u^i}^\alpha R_\vartheta^\beta\delta_{\alpha\beta}
            +R_\sigma^\alpha R_\vartheta^\beta\widetilde T_{\alpha,i}^\beta\big)
           \,=\,\sum_{\alpha=1}^n
                R_{\sigma,u^i}^\alpha (R^t)_\alpha^\vartheta
                 +\sum_{\alpha,\beta=1}^nR_\sigma^\alpha\widetilde T_{\alpha,i}^\beta(R^t)_\beta^\vartheta
    \end{array}$$
due to $R_\vartheta^\alpha=(R^t)_\alpha^\vartheta.$ Thus, we arrive at the concise transformation rule
\begin{equation}\label{transform_T}
  T_i=R_{u^i}\circ R^t+R\circ\widetilde T_i\circ R^t\,.
\end{equation}
Using this formula we now evaluate $S_{12}=T_{1,v}-T_{2,u}-T_1\circ T_2^t+T_2\circ T_1^t:$ First
  $$\begin{array}{lll}
      T_{1,v}-T_{2,u}
      & = & \displaystyle\negthickspace
            (R_u\circ R^t+R\circ\widetilde T_1\circ R^t)_v
            -(R_v\circ R^t+R\circ\widetilde T_2\circ R^t)_u \\[2ex]
      & = & \displaystyle\negthickspace
            R_u\circ R_v^t-R_v\circ R_u^t
            +R\circ(\widetilde T_{1,v}-\widetilde T_{2,u})\circ R^t \\[2ex]
      &   & \displaystyle\negthickspace
            +\,R_v\circ\widetilde T_1\circ R^t
            +R\circ\widetilde T_1\circ R_v^t
            -R_u\circ\widetilde T_2\circ R^t
            -R\circ\widetilde T_2\circ R_u^t\,,
    \end{array}$$
and next
  $$\begin{array}{lll}
      T_1\circ T_2^t-T_2\circ T_1^t\negthickspace
      & = & \displaystyle\negthickspace
            (R_u\circ R^t+R\circ\widetilde T_1\circ R^t)
            \circ(R\circ R_v^t+R\circ\widetilde T_2^t\circ R^t) \\[2ex]
      &   & \displaystyle\negthickspace
            -\,(R_v\circ R^t+R\circ\widetilde T_2\circ R^t)
               \circ(R\circ R_u^t+R\circ\widetilde T_1^t\circ R^t) \\[2ex]
      & = & \displaystyle\negthickspace
            R_u\circ R_v^t
            +R_u\circ\widetilde T_2^t\circ R^t
            +R\circ\widetilde T_1\circ R_v^t
            +R\circ\widetilde T_1\circ\widetilde T_2^t\circ R^t \\[2ex]
      &   & \displaystyle\negthickspace
            -\,R_v\circ R_u^t
            -R_v\circ\widetilde T_1^t\circ R^t
            -R\circ\widetilde T_2\circ R_u^t
            -R\circ\widetilde T_2\circ\widetilde T_1^t\circ R^t\,
    \end{array}$$
on account of $R\circ R^t=R^t\circ R=\mbox{id}.$ Taking both identities together, we arrive at
  $$\begin{array}{l}
      T_{1,v}-T_{2,u}-T_1\circ T_2^t+T_2\circ T_1^t \\[2ex]
      \hspace*{3ex}\displaystyle
      =\,R\circ
         (\widetilde T_{1,v}-\widetilde T_{2,u}-
         \widetilde T_1\circ\widetilde T_2^t+\widetilde T_2\circ\widetilde T_1^t)
         \circ R^t \\[2ex]
      \hspace*{6ex}\displaystyle
      +\,R_v\circ\widetilde T_1\circ R^t+R\circ\widetilde T_1\circ R_v^t
      -R_u\circ\widetilde T_2\circ R^t-R\circ\widetilde T_2\circ R_u^t \\[2ex]
      \hspace*{6ex}\displaystyle
      -\,R_u\circ\widetilde T_2^t\circ R^t-R\circ\widetilde T_1\circ R_v^t
      +R_v\circ\widetilde T_1^t\circ R^t+R\circ\widetilde T_2\circ R_u^t \\[2ex]
      \hspace*{3ex}\displaystyle
      =\,R\circ
         (\widetilde T_{1,v}-\widetilde T_{2,u}
         -\widetilde T_1\circ\widetilde T_2^t+\widetilde T_2\circ\widetilde T_1^t)
         \circ R^t
    \end{array}$$
due to $\widetilde T_i=-\widetilde T_i^t$. This proves the statement.
\end{proof}
\begin{remark}
\label{remark_w21}
Note that the above calculations remain valid for rotations $R\in W_{loc}^{2,1}(B,SO(n))\cap W^{1,2}(B,SO(n))$. 
\end{remark}
\noindent
The transformation rule (\ref{transform_S}) gives rise to the definition of a geometric curvature quantity associated with the normal bundle of surfaces:
\begin{corollary}\label{c1}
The \emph{normal curvature vector}\footnote{Note that $(N_\sigma\wedge N_\vartheta)_{1\le\sigma<\vartheta\le n}$ forms a basis in the Grassmanian space $\mathbb R^\frac{n(n+1)}{2}$; thus we can interprete $\mathcal S$ as a vector in $\mathbb R^\frac{n(n+1)}{2}$ with components $(2S_{\sigma,12}^\vartheta)_{1\le\sigma<\vartheta\le n}$. We refer to \cite{Heil_01} for details.}
  $${\mathcal S}
    :=\frac{1}{W}\,
      \sum_{\sigma,\vartheta=1}^n
      S_{\sigma,12}^\vartheta\,N_\sigma\wedge N_\vartheta$$
is invariant w.r.t.~rotations and positively oriented parameter transformations. Here, $\wedge$ denotes the outer product or wedge product between vectors in $\mathbb R^{n+2},$ and $W$ is the area element of the immersion $X.$
\end{corollary}
\begin{proof}
Let us first check the invariance w.r.t.~rotations: From (\ref{transform_S}) we infer
\begin{equation*}
\begin{array}{lll}
  \displaystyle
  \sum_{\sigma,\vartheta=1}^n
  S_{\sigma,12}^\vartheta\,N_\sigma\wedge N_\vartheta
  & = & \displaystyle
        \sum_{\sigma,\vartheta=1}^n
        \sum_{\alpha,\beta=1}^n
        S_{\sigma,12}^\vartheta R_\sigma^\alpha R_\vartheta^\beta\,\widetilde N_\alpha\wedge\widetilde N_\beta \\[4ex]
  & = & \displaystyle
        \sum_{\sigma,\vartheta=1}^n
        \sum_{\alpha,\beta=1}^n
        (R^t)_\alpha^\sigma S_{\sigma,12}^\vartheta R_\vartheta^\beta\,\widetilde N_\alpha\wedge\widetilde N_\beta \\[4ex]
  & = & \displaystyle
        \sum_{\alpha,\beta=1}^n
        \widetilde S_{\alpha,12}^\beta\,\widetilde N_\alpha\wedge\widetilde N_\beta\,.
\end{array}
\end{equation*}
Now we verify the parameter invariance of ${\mathcal S}:$ Let $u^i(\widehat u^m),$ $i=1,2$ and $m=1,2,$ be a positively oriented parameter transformation. Note that by construction it does not affect the normal frame. We should compute the transformation taking the Ricci integrability conditions (see e.g. Chen \cite{chen_1973})
  $$S_{\sigma,ij}^\vartheta
    =\sum_{k,\ell=1}^2
     (L_{\sigma,ik}L_{\vartheta,j\ell}-L_{\sigma,jk}L_{\vartheta,i\ell})g^{k\ell}$$
into account; here $L_{\sigma,ik}:=-\langle N_{\sigma,u^i},X_{u^k}\rangle$ denote the coefficients of the second fundamental form w.r.t.~$N_\sigma$. Writing $\widehat S_{\sigma,mp}^\vartheta$ for the normal curvature tensor computed in the parameters $\widehat u^m$, we arrive at the transformation rule
  $$S_{\sigma,ij}^\vartheta
    =\sum_{m,p=1}^2
     \widehat S_{\sigma,mp}^\vartheta\,\frac{\partial\widehat u^m}{\partial u^i}\,\frac{\partial\widehat u^p}{\partial u^j}\,,\quad
    S_{\sigma,12}^\vartheta
    =\widehat S_{\sigma,12}^\vartheta
     \left(
       \frac{\partial\widehat u^1}{\partial u^1}\,\frac{\partial\widehat u^2}{\partial u^2}
       -\frac{\partial\widehat u^1}{\partial u^2}\,\frac{\partial\widehat u^2}{\partial u^1}
     \right).$$
Thus,  $W^{-1}S_{\sigma,12}^\vartheta$ and $\mathcal S$ are invariant w.r.t.~such parameter transformations.
\end{proof}
\noindent
The squared length $|{\mathcal S}|^2$ of the normal curvature vector is usually called the {\it normal curvature of the surface,} see e.g.~\cite{chen_ludden_1972}. It seems promising to us to study {\it surfaces with prescribed normal curvature vector} ${\mathcal S}$ in analogy to surfaces with prescribed mean curvature vector. 
\subsection{Existence of weak normal Coulomb frames}\label{ss4.1}
In \cite{helein_2002} Lemma 4.1.3, H\'elein proved existence of weak Coulomb frames in the tangential bundle of a manifold. His method can be adapted to our situation. For reasons of completeness we carry out the arguments.
\begin{proposition}\label{prop_exist}
There exists a weak normal Coulomb frame $N\in W^{1,2}(B)\cap L^\infty(B)$ minimizing the functional $\mathcal T_X$ of total torsion in the set of all weak normal frames of class $W^{1,2}(B)\cap L^\infty(B)$. 
\end{proposition}
\begin{proof}
We fix some normal frame $\widetilde N\in C^{k-1,\alpha}(\overline B)$ and interpret $\mathcal T_X$ as a functional ${\mathcal F}(R)$ of rotations $R=(R_\sigma^\vartheta)_{\sigma,\vartheta=1,\ldots,n}\in W^{1,2}(B,SO(n))$ by setting
  $$\mathcal F(R)=\sum_{\sigma,\vartheta=1}^n\sum_{i=1}^2\int\hspace{-1.3ex}\int\limits_{\hspace{-1.8ex}B}
    (T_{\sigma,i}^\vartheta)^2\,dudv=\int\hspace{-1.3ex}\int\limits_{\hspace{-1.8ex}B}\big(|T_1|^2+|T_2|^2\big)\,dudv,
    \quad N_\sigma:=\sum_{\vartheta=1}^nR_\sigma^\vartheta\widetilde N_\vartheta.$$ 
Choose a minimizing sequence ${}^{\scriptscriptstyle \ell}\!R=({}^{\scriptscriptstyle \ell}\!R_\sigma^\vartheta)_{\sigma,\vartheta=1,\ldots,n}\in W^{1,2}(B,SO(n))$ and define ${}^{\scriptscriptstyle \ell}\!N_\sigma:=\sum\limits_{\vartheta=1}^n{}^{\scriptscriptstyle \ell}\!R_\sigma^\vartheta\widetilde N_\vartheta$. As in (\ref{transform_T}) we find ${}^{\scriptscriptstyle \ell}T_i={}^{\scriptscriptstyle \ell}\!R_{u^i}\circ{}^{\scriptscriptstyle \ell}\!R^t+{}^{\scriptscriptstyle \ell}\!R\circ \widetilde T_i\circ{}^{\scriptscriptstyle \ell}\!R^t,$ and this implies
\begin{equation*}
\begin{array}{lll}
  {}^{\scriptscriptstyle \ell}T_i\circ{}^{\scriptscriptstyle \ell}T_i^t
  & = & \displaystyle
        ({}^{\scriptscriptstyle \ell}\!R_{u^i}\circ
         {}^{\scriptscriptstyle \ell}\!R^t
         +{}^{\scriptscriptstyle \ell}\!R\circ
          \widetilde T_i\circ
          {}^{\scriptscriptstyle \ell}\!R^t)\circ
        ({}^{\scriptscriptstyle \ell}\!R\circ
         {}^{\scriptscriptstyle \ell}\!R_{u^i}^t
         +{}^{\scriptscriptstyle \ell}\!R\circ
          \widetilde T_i^t\circ
          {}^{\scriptscriptstyle \ell}\!R^t) \\[2ex]
  & = & \displaystyle
        {}^{\scriptscriptstyle \ell}\!R_{u^i}\circ
        {}^{\scriptscriptstyle \ell}\!R_{u^i}^t
        +{}^{\scriptscriptstyle \ell}\!R\circ
         \widetilde T_i\circ
         {}^{\scriptscriptstyle \ell}\!R_{u^i}^t
        +{}^{\scriptscriptstyle \ell}\!R_{u^i}\circ
         \widetilde T_i^t\circ
         {}^{\scriptscriptstyle \ell}\!R^t
        +{}^{\scriptscriptstyle \ell}\!R\circ
         \widetilde T_i\circ
         \widetilde T_i^t\circ
         {}^{\scriptscriptstyle \ell}\!R^t.
\end{array}
\end{equation*}
In particular, we conclude
\begin{equation*}
  \mbox{trace}\,({}^{\scriptscriptstyle \ell}T_i
  \circ{}^{\scriptscriptstyle \ell}T_i^t)
  =\mbox{trace}\,
   ({}^{\scriptscriptstyle \ell}\!R_{u^i}
   \circ{}^{\scriptscriptstyle \ell}\!R_{u^i}^t)
   +2\,\mbox{trace}\,
    ({}^{\scriptscriptstyle \ell}\!R
    \circ \widetilde T_i\circ{}^{\scriptscriptstyle \ell}\!R_{u^i}^t)
   +\mbox{trace}\,(\widetilde T_i\circ \widetilde T_i^t)
\end{equation*}
or
\begin{equation}\label{torsion_sequence}
  |{}^{\scriptscriptstyle \ell}T_i|^2
  =|{}^{\scriptscriptstyle \ell}\!R_{u^i}|^2
   +2\langle{}^{\scriptscriptstyle \ell}\!R\circ\widetilde T_i,{}^{\scriptscriptstyle \ell}\!R_{u^i}\rangle
   +|\widetilde T_i|^2.
\end{equation}
Taking $|{}^{\scriptscriptstyle \ell}\!R\circ\widetilde T_i|=|\widetilde T_i|$ into account, we arrive at the estimate
\begin{equation}\label{estimate_r}
  |{}^{\scriptscriptstyle \ell}T_i|^2
  \ge\big(|\widetilde T_i|-|{}^{\scriptscriptstyle \ell}\!R_{u^i}|\big)^2\quad\mbox{a.e.~on}\ B, 
  \quad\mbox{for all}\ \ell\in\mathbb N.
\end{equation}
Now the $\widetilde T_i$ are bounded in $L^2(B).$ And since ${}^{\scriptscriptstyle \ell}\!R$ is minimizing for $\mathcal F$, the sequences ${}^{\scriptscriptstyle \ell}T_i$ are also bounded in $L^2(B).$ Thus, ${}^{\scriptscriptstyle \ell}\!R_{u^i}$ are bounded sequences in $L^2(B)$ in accordance with (\ref{estimate_r}). By Hilbert's selection theorem and Rellich's embedding theorem we find a subsequence, again denoted by ${}^{\scriptscriptstyle \ell}\!R$, which converges as follows:
\begin{equation*}
  {}^{\scriptscriptstyle \ell}\!R_{u^i}
  \rightharpoonup R_{u^i}
  \quad\mbox{weakly in}\ L^2(B,\mathbb R^{n\times n}),\quad
  {}^{\scriptscriptstyle \ell}\!R\rightarrow R
  \quad\mbox{strongly in}\ L^2(B,SO(n))
\end{equation*}
with some $R\in W^{1,2}(B,SO(n))$. In particular, we have ${}^{\scriptscriptstyle \ell}\!R\to R$ a.e.~on $B$ and 
  $$\lim_{\ell\to\infty}
	\int\hspace{-1.3ex}\int\limits_{\hspace{-1.8ex}B}
    |{}^{\scriptscriptstyle \ell}\!R\circ \widetilde T_i-R\circ\widetilde T_i|^2\,dudv=0$$
according to the dominated convergence theorem. Hence, we can compute in the limit
\begin{equation*}
\begin{array}{rcl}
  \displaystyle\lim_{\ell\to\infty}
  \int\hspace{-1.3ex}\int\limits_{\hspace{-1.8ex}B}
  \langle
    {}^{\scriptscriptstyle \ell}\!R\circ\widetilde T_i,{}^{\scriptscriptstyle \ell}\!R_{u^i}
  \rangle\,dudv
   & = & \displaystyle\lim_{\ell\to\infty}
  \Bigg(\!\int\hspace{-1.3ex}\int\limits_{\hspace{-1.8ex}B}
  \langle
    {}^{\scriptscriptstyle \ell}\!R\circ\widetilde T_i-R\circ\widetilde T_i,{}^{\scriptscriptstyle \ell}\!R_{u^i}
  \rangle\,dudv
  +\int\hspace{-1.3ex}\int\limits_{\hspace{-1.8ex}B}
     \langle R\circ\widetilde T_i,{}^{\scriptscriptstyle \ell} R_{u^i}
     \rangle\,dudv\!\Bigg)\\[4ex]
   & = & \displaystyle  
   \int\hspace{-1.3ex}\int\limits_{\hspace{-1.8ex}B}
     \langle R\circ\widetilde T_i,R_{u^i}
     \rangle\,dudv.
\end{array}
\end{equation*}
In addition, we obtain
\begin{equation*}
  \lim_{\ell\to\infty}
  \int\hspace{-1.3ex}\int\limits_{\hspace{-1.8ex}B}
  |{}^{\scriptscriptstyle \ell}\!R_{u^i}|^2\,dudv
  \ge\int\hspace{-1.3ex}\int\limits_{\hspace{-1.8ex}B}
     |R_{u^i}|^2\,dudv
\end{equation*}
due to the semicontinuity of the $L^2$-norm w.r.t.~weak convergence. Putting the last two relations into (\ref{torsion_sequence}), we finally infer
\begin{equation*}
  \lim_{\ell\to\infty}\mathcal F({}^{\scriptscriptstyle \ell}\!R)
  =\lim_{\ell\to\infty}
  \int\hspace{-1.3ex}\int\limits_{\hspace{-1.8ex}B}
  \big(|{}^{\scriptscriptstyle \ell}T_1|^2+|{}^{\scriptscriptstyle \ell}T_2|^2\big)\,dudv
  \ge\int\hspace{-1.3ex}\int\limits_{\hspace{-1.8ex}B}
  \big(|T_1|^2+|T_2|^2\big)\,dudv=\mathcal F(R),
\end{equation*}
where $T_i=(T_{\sigma,i}^\vartheta)_{\sigma,\vartheta=1,\ldots,n}$ denote the torsion coefficients of the frame $N$ with entries $N_\sigma:=\sum_\vartheta R_\sigma^\vartheta\widetilde N_\vartheta$ (note that the calculations leading to (\ref{torsion_sequence}) yield an analogous relation for $|T_i|^2$). Consequently, $N\in W^{1,2}(B)\cap L^\infty(B)$ minimizes $\mathcal T_X$ and, in particular, is a weak normal Coulomb frame; compare with Remark \ref{weak_coulomb}. 
\end{proof}
\subsection{Regularity of weak normal Coulomb frames}\label{ss4.2}
In order to prove our main existence result, Theorem \ref{existence_rn} below, it remains to show the smoothness of the weak normal Coulomb frame constructed in Proposition \ref{prop_exist}. We start with the following
\begin{lemma}\label{w_12}
Any weak normal Coulomb frame $N\in W^{1,2}(B)\cap L^\infty(B)$ belongs to the class $W^{2,1}_{loc}(B)$.
\end{lemma}
\begin{proof}
\begin{enumerate}
\item
The torsion coefficients $T_{\sigma,i}^\vartheta$ of the normal Coulomb frame $N$ are weak solutions of the Euler-Lagrange equations
\begin{equation*}
  \mbox{div}\,(T_{\sigma,1}^\vartheta,T_{\sigma,2}^\vartheta)=0\quad\mbox{in}\ B,\quad
  \big\langle(T_{\sigma,1}^\vartheta,T_{\sigma,2}^\vartheta),\nu\big\rangle=0\quad\mbox{on}\ \partial B
\end{equation*}
for all $\sigma,\vartheta=1,\ldots,n$; see Proposition \ref{p1} and Remark \ref{weak_coulomb}. Hence, by a weak version of Poincare's lemma (see e.g.~\cite{bourgain_2000} Lemma 3), there are integral functions $\tau_\sigma^\vartheta\in W^{1,2}(B)$ satisfying
\begin{equation}\label{need_b}
  \tau_{\sigma,u}^\vartheta=-T_{\sigma,2}^\vartheta\,,\quad
  \tau_{\sigma,v}^\vartheta=T_{\sigma,1}^\vartheta
  \quad\mbox{in}\ B\,.
\end{equation}
We now may write the weak form (\ref{weak_euler}) of the Euler-Lagrange equations as
  $$0=\int\hspace{-1.3ex}\int\limits_{\hspace{-1.8ex}B}\big\{\varphi_{u^1}\tau_{\sigma,u^2}^\vartheta-\varphi_{u^2}\tau_{\sigma,u^1}^\vartheta\big\}
     \,dudv=\int\limits_{\partial B}\tau_\sigma^\vartheta\frac{\partial\varphi}{\partial t}\,ds\quad\mbox{for all}\ \varphi\in C^\infty(\overline B),$$
where $\frac{\partial\varphi}{\partial t}$ denotes the tangential derivative of $\varphi$ along $\partial B$ and, as usual, we have written $\tau_\sigma^\vartheta$ for the $L^2$-trace of $\tau_\sigma^\vartheta$ on $\partial B$. Consequently, the lemma of DuBois-Reymond yields $\tau_\sigma^\vartheta\equiv\mbox{const}$ on $\partial B$, and by translation we arrive at the boundary conditions
\begin{equation}\label{need_b2}
  \tau_\sigma^\vartheta=0\quad\mbox{on}\ \partial B\,.
\end{equation}
\item
As can be seen by approximation, the system (\ref{need_b}) and the definition of the torsion coefficients $T_{\sigma,i}^\vartheta$ imply that the $\tau_\sigma^\vartheta$ are weak solutions of the second-order system
  $$\Delta\tau_\sigma^\vartheta
    =-T_{\sigma,2,u}^\vartheta+T_{\sigma,1,v}^\vartheta
    =-\langle N_{\sigma,v},N_{\vartheta,u}\rangle
     +\langle N_{\sigma,u},N_{\vartheta,v}\rangle\quad\mbox{in}\ B\,.$$
By a result of S.\,M\"uller \cite{mueller_1990} and Coifman, Lions, Meyer and Semmes \cite{coifman_lions_meyer_semmes_1993}, the right-hand side of div-curl type belongs to the Hardy space $\mathcal H_{loc}^1(B)$ and, hence, the $\tau_\sigma^\vartheta$ belong to $W_{loc}^{2,1}(B)$ by Fefferman and Stein \cite{fefferman_stein_1972}. Consequently, we find $T_{\sigma,i}^\vartheta\in W^{1,1}_{loc}(B)\cap L^2(B).$ Next, we employ the Weingarten equations (see e.g.~Chen \cite{chen_1973})
\begin{equation}\label{weingarten}
  N_{\sigma,u^i}=-\sum_{j,k=1}^2L_{\sigma,ij}g^{jk}\,X_{u^k}+\sum_{\vartheta=1}^nT_{\sigma,i}^\vartheta N_\vartheta
\end{equation}
in a weak form. For the coefficients of the second fundamental form we have $L_{\sigma,ij}=\langle N_\sigma,X_{u^iu^j}\rangle$ and, thus, $L_{\sigma,ij}\in W^{1,2}(B)$ taking account of $N\in W^{1,2}(B).$ Hence, the Weingarten equations (\ref{weingarten}) yield $N_{\sigma,u^i}\in W_{loc}^{1,1}(B)$ and $N\in W_{loc}^{2,1}(B)$ for our weak Coulomb frame\footnote{Note that $T_{\sigma,i}^\vartheta\in W^{1,1}_{loc}(B)\cap L^2(B)$ and $N_\vartheta\in W^{1,2}(B)\cap L^\infty(B)$ imply $T_{\sigma,i}^\vartheta N_\vartheta\in W^{1,1}_{loc}(B)$ by a careful adaption of the classical product rule in Sobolev spaces.}. This proves the lemma.
\vspace*{-3ex}
\end{enumerate}
\end{proof}
\goodbreak
\begin{theorem}
\label{existence_rn}
For any conformally parametrized immersion $X\in C^{k,\alpha}(\overline B,\mathbb R^{n+2})$ with $k\ge3$ and $\alpha\in(0,1)$ there exists a normal Coulomb frame $N\in C^{k-1,\alpha}(\overline B,\mathbb R^{(n+2)\times n})$ minimizing $\mathcal T_X$.
\end{theorem}
\begin{proof}
\begin{itemize}
\item[1.]
We fix some normal frame $\widetilde N\in C^{k-1,\alpha}(\overline B)$ and construct a weak normal Coulomb frame $N\in W^{1,2}(B)\cap L^\infty(B)$ by Proposition \ref{prop_exist}. Due to Lemma\,\ref{w_12} we then know $N\in W^{2,1}_{loc}(B)$. Defining the orthogonal mapping $R=(R_\sigma^\vartheta)_{\sigma,\vartheta=1,\ldots,n}$ by $R_\sigma^\vartheta:=\langle N_\sigma,\widetilde N_\vartheta\rangle,$
we thus find 
	$$N_\sigma=\sum\limits_{\vartheta=1}^nR_\sigma^\vartheta\widetilde N_\vartheta\quad\mbox{and}\quad R\in W^{2,1}_{loc}(B,SO(n))
      \cap W^{1,2}(B,SO(n))\,.$$
In particular, we can assign a curvature tensor $S_{12}=(S_{\sigma,12}^\vartheta)_{\sigma,\vartheta=1,\ldots,n}\in L^1_{loc}(B)$ to $N$ by formula (\ref{s_12}), and from Theorem \ref{curvature_vector} we conclude $S_{12}\in L^\infty(B)$; compare also Remark \ref{curvature_vector}.
\item[2.]
Introduce $\tau=(\tau_\sigma^\vartheta)_{\sigma,\vartheta=1,\ldots,n}\in W^{1,2}(B)$ by (\ref{need_b}), (\ref{need_b2}). The definition of the normal curvature tensor gives us the nonlinear elliptic system
\begin{equation}\label{need_c}
  \Delta\tau_\sigma^\vartheta
  =-\,\sum_{\omega=1}^n
    (\tau_{\sigma,u}^\omega\tau_{\omega,v}^\vartheta-\tau_{\sigma,v}^\omega\tau_{\omega,u}^\vartheta)
   +S_{\sigma,12}^\vartheta\quad\mbox{in}\ B,\quad
  \tau_\sigma^\vartheta=0\quad\mbox{on}\ \partial B.
\end{equation}
On account of $S_{12}=(S_{\sigma,12}^\vartheta)_{\sigma,\vartheta=1,\ldots,n}\in L^\infty(B),$ a part of Wente's inequality yields $\tau\in C^0(\overline B),$ see e.g.~\cite{brezis_coron_1984}; compare also Rivi\`ere \cite{riviere_2007} and the corresponding boundary regularity theorem in M\"uller and Schikorra \cite{mueller_schikorra_2008} for more general results. By appropriate reflection of $\tau$ and 
$S_{12}$ (the reflected quantities are again denoted by $\tau$ and $S_{12}$) we obtain a weak solution $\tau\in W^{1,2}(B_{1+d})\cap C^0(B_{1+d})$ of
\begin{equation}\label{quadratic_1}
  \Delta\tau
  =f(w,\nabla\tau)
  \quad\mbox{in}\ B_{1+d}:=\{w\in\mathbb R^2\,:\ |w|<1+d\}
\end{equation}
with some $d>0$ and a right-hand side $f$ satisfying
\begin{equation}\label{quadratic_2}
  |f(w,p)|\le a|p|^2+b\quad\mbox{for all}\ p\in\mathbb R^{2n^2},\ w\in B_{1+d}
\end{equation}
with some reals $a,b>0.$ Now, applying Tomi's regularity result \cite{tomi} for weak solutions of the system (\ref{quadratic_1}), (\ref{quadratic_2}) possessing small variation locally in $B_{1+d}$, we find $\tau\in C^{1,\nu}(\overline B)$ for any $\nu\in(0,1)$ (note that Tomi's result applies for such systems with $b=0,$ but his proof can easily be adapted to our inhomogeneous case $b>0$).
\item[3.]
From (\ref{need_b}) we infer $T_i\in C^\alpha(\overline B).$ Thus, the Weingarten equations (\ref{weingarten}) yield $N\in W^{1,\infty}(B)$ on account of $N\in L^\infty(B),$ and we obtain $N\in C^\alpha(\overline B)$ by Sobolev's embedding theorem. Inserting 
this again into the Weingarten equations, we find $N\in C^{1,\alpha}(\overline B)$. Hence, we can conclude $R=(\langle N_\sigma,\widetilde N_\vartheta\rangle)_{\sigma,\vartheta=1,\ldots,n}\in C^{1,\alpha}(\overline B),$ and Theorem \ref{curvature_vector} implies $S_{12}=(S_{\sigma,12}^\vartheta)_{\sigma,\vartheta=1,\ldots,n}\in C^{\alpha}(\overline B)$ (note $\widetilde S_{12}\in C^\alpha(\overline B)$ for $k=3$; in case $k\ge 4$ we even get $S_{12}\in C^{1,\alpha}(\overline B)$). Now the right-hand side of (\ref{need_c}) belongs to $C^\alpha(\overline B),$ and potential theoretic estimates ensure $\tau\in C^{2,\alpha}(\overline B)$. Involving again (\ref{need_b}) gives $T_i\in C^{1,\alpha}(\overline B),$ which proves $N\in C^{2,\alpha}(\overline B)$ using the Weingarten equations once more. Finally, for $k\ge4$, we can bootstrap by concluding $R\in C^{2,\alpha}(\overline B)$ and $S_{12}\in C^{1,\alpha}(\overline B)$ from Theorem \ref{curvature_vector} and repeating the
  arguments above. 
\vspace*{-3ex}
\end{itemize}
\end{proof}
\subsection{An a priori estimate for the torsion coefficients}
To illustrate the advantages of working with normal Coulomb frames we want to present the reader a new analytical estimate for the torsion coefficients of those frames.\\[1ex]
Let $X\in C^{k,\alpha}(\overline B,\mathbb R^{n+2})$ be a conformally parametrized immersion with $k\ge 3$, $\alpha\in(0,1)$. Then we can define the a priori constants
  $$\mathcal T_0:=\inf_{\mbox{\scriptsize $\widetilde N$\,is normal frame}}\mathcal T_X(\widetilde N),\quad S_0:=\sup_{w\in B}|S_{12}(w)|.$$
Here the normal curvature tensor $S_{12}=(S_{\sigma,12}^\vartheta)_{\sigma,\vartheta=1,\ldots,n}$ can be computed from any normal frame $N$, according to Theorem \ref{curvature_vector}. As an easy consequence of our prior work in \cite{froehlich_mueller_2008} we now obtain the following
\begin{theorem}\label{a_priori}
Consider some conformally parametrized immersion $X\in C^{k,\alpha}(\overline B,\mathbb R^{n+2})$ with $k\ge3$, $\alpha\in(0,1)$, and assume the smallness condition
\begin{equation}\label{small}
  \frac{\sqrt{n-2}}4
  \left(
    \frac{n-2}{2\pi}\,\mathcal T_0+\gamma(n)S_0\right)<1
\end{equation}
to be satisfied with $\gamma(n):=\min\{\frac14\sqrt{\frac{n(n-1)}2},\sqrt2\}$. Then there exists a normal Coulomb frame $N\in C^{k-1,\alpha}(\overline B)$ such that the torsion coefficients $T_1,T_2$ of $N$ can be estimated by
\begin{equation}\label{estimate}
  |T_i(w)|\le c\quad\mbox{for all}\ w\in\overline B,\ i=1,2
\end{equation}
with an a priori constant $c=c(n,\mathcal T_0,S_0)<+\infty$.
\end{theorem}
\begin{proof}
In virtue of Theorem \ref{existence_rn} there exists a normal Coulomb frame $N\in C^{k-1,\alpha}(\overline B)$ with $\mathcal T_X(N)=\mathcal T_0$. Hence, the theorem follows directly from Theorem 3 in \cite{froehlich_mueller_2008} applied to $N$.
\end{proof}

\noindent{\bf Acknowledgement:} We are grateful to Ulrich Dierkes and Ruben Jakob for helpful discussions during the preparation of this manuscript.

\vspace*{4ex}
\noindent
\begin{tabular}{ll}
Steffen Fr\"ohlich                                 &    Frank M\"uller\\
Freie Universit\"at Berlin                         &    Universit\"at Duisburg-Essen\\
Fachbereich Mathematik und Informatik\hspace{8ex}  &    Campus Duisburg, Fachbereich Mathematik\\
Arnimallee 2-6                                     &    Lotharstrasse 65\\
D-14195 Berlin                                     &    D-47048 Duisburg\\
Germany                                            &    Germany\\[0.6ex]
e-mail: sfroehli@mi.fu-berlin.de                   &    e-mail: mueller@math.tu-cottbus.de
\end{tabular}

\end{document}